\documentclass[10pt,letterpaper]{article}
\usepackage[T1]{fontenc}
\usepackage{lmodern,amsmath,amssymb,amsthm,mathtools}
\usepackage[margin=30mm]{geometry}
\usepackage{microtype,enumitem,longtable,booktabs}
\usepackage{tocloft,fancyhdr}
\usepackage[hidelinks,hypertexnames=false]{hyperref}
\hypersetup{pdftitle={Finite-time blow-up for the Geng-Xue system with an application to the three-component Degasperis-Procesi system},pdfauthor={Song Liu and Zhaoyang Yin}}
\newcommand{\R}{\mathbb R}
\newcommand{\T}{\mathbb T}

\newcommand{\dd}{\,\mathrm d}
\newcommand{\cothop}{\operatorname{coth}}
\DeclareMathOperator{\supp}{supp}
\newtheorem{theo}{Theorem}[section]
\newtheorem{rema}[theo]{Remark}
\newtheorem{prop}[theo]{Proposition}
\newtheorem{lem}[theo]{Lemma}
\theoremstyle{definition}
\newtheorem{defi}[theo]{Definition}
\newtheorem{example}{Example}[section]
\numberwithin{equation}{section}
\allowdisplaybreaks[2]
\begin{document}
\title{Finite-time blow-up for the Geng--Xue system
with an application to the three-component
Degasperis--Procesi system
\hspace{-4mm}
}
\author{Song $\mbox{Liu}^1$ \footnote{Email: lius37@mail2.sysu.edu.cn} \quad and\quad	
Zhaoyang $\mbox{Yin}^1$\footnote{E-mail: mcsyzy@mail.sysu.edu.cn}\\	
$^1\mbox{School}$ of Science,\\ Shenzhen Campus of Sun Yat-sen University, Shenzhen, 518107, China}
\date{}
\maketitle
\hrule
\begin{abstract}
We prove finite-time blow-up for the Geng--Xue system with initial data
satisfying $u_0\in W^{p,q}(\R)$ and $n_0\in L^{r'}(\R)$ at Sobolev
regularity $s\geq2$, where $n=v-v_{xx}$. A cancellation in the balance
law for $nu_x$ yields closed convolution estimates that do not depend
on an a priori uniform bound for $u_x$. These estimates control $v$
along a characteristic and give an explicit condition on the initial
slope, together with an upper bound on the lifespan. The proof does not
use conservation laws. For the three-component Degasperis--Procesi system,
the same argument applies to $nu_x+\rho^2-\bar\rho^2$, where
$\bar\rho\in[0,1]$ is a constant density background. We also prove local
well-posedness for this system in $H^{s-1}(\R)\times H^s(\R)\times H^s(\R)$
for $s>\frac32$, and a blow-up criterion in
$B^1_{2,1}(\R)\times B^2_{2,1}(\R)\times B^2_{2,1}(\R)$.
\end{abstract}


{\small\setlength{\parskip}{0pt}\tableofcontents}
\clearpage

\section{Introduction}
We consider the Cauchy problem for the Geng--Xue (GX) system
\cite{geng2009extension}:
\begin{equation}\label{GX}\tag{1.1}
\begin{cases}
m_t+uvm_x+3vu_xm=0,\\
n_t+uvn_x+3uv_xn=0,\\
m=u-u_{xx},\qquad n=v-v_{xx}.
\end{cases}
\end{equation}
This system was introduced by Geng and Xue \cite{geng2009extension}, who
established its Hamiltonian structure and showed that it admits peakon
solutions. Lundmark and Szmigielski \cite{lundmark2016inverse} solved the
associated spectral and inverse spectral problems. Concerning the Cauchy
problem, Himonas and Mantzavinos \cite{himonas2016novikov} established local
well-posedness in $H^s\times H^s$ for $s>\frac32$ and proved that the
data-to-solution map is not uniformly continuous. Chen, Qiao, and Zhou
\cite{chen2019persistence} investigated persistence properties and established
a blow-up criterion. More recently, Liu and Yin \cite{liu2026} established a new
blow-up criterion in a low-regularity Besov space and proved finite-time blow-up
under suitable pointwise conditions on the initial data, using a method that
does not rely on conservation laws.

When $v\equiv1$, system \eqref{GX} reduces to the Degasperis--Procesi (DP)
equation \cite{degasperis1999asymptotic}:
\[
m_t+um_x+3u_xm=0,\qquad m=u-u_{xx}.
\]
The DP equation serves as a model for nonlinear shallow water waves
\cite{constantin2010inverse,constantin2009hydrodynamical}. It possesses a
bi-Hamiltonian structure, infinitely many conservation laws, and peakon
solutions \cite{degasperis2002new}. These features are also shared by the
Camassa--Holm (CH) equation \cite{camassa1993integrable,camassa1994new}, which,
in its dispersionless form, reads
\[
m_t+um_x+2u_xm=0,\qquad m=u-u_{xx}.
\]
For results on the Cauchy problem for the CH equation, see
\cite{bressan2007global,danchin2001remarks}.

Qiao
\cite{qiao2004integrable} constructed an integrable hierarchy admitting
$3\times3$ matrix Lax pairs and obtained parametric solutions of the DP
equation through a reduction to finite-dimensional integrable Hamiltonian
systems on a symplectic submanifold. Subsequently, Hou et al.
\cite{hou2013algebro} constructed algebro-geometric solutions for the DP
hierarchy. For the Cauchy and initial-boundary value problems for the DP equation,
see \cite{coclite2006well,escher2009initial,yin2003cauchy,yin2004global}.

When $v\equiv u$, system \eqref{GX} reduces to the Novikov equation, introduced
by Novikov \cite{novikov2009generalizations}:
\[
m_t+u^2m_x+3uu_xm=0,\qquad m=u-u_{xx}.
\]
This integrable equation has cubic nonlinearity and admits peakon solutions \cite{MR2430566}. Its well-posedness, blow-up phenomena, and
ill-posedness have been investigated in
\cite{himonas2012cauchy,himonas2013holder,Li2025}.

For the full GX system, we estimate the two components jointly through
$nu_x$. This gives the bounds needed for the blow-up argument without
using a conservation law.

We also consider the following three-component Degasperis--Procesi
(3-DP) system:
\begin{equation}\label{(1.1)}\tag{1.2}
\begin{cases}
\rho_t+(\rho uv)_x=0,\\
m_t+uvm_x+3vu_xm+\rho^2u=0,\\
n_t+uvn_x+3uv_xn-\rho^2v=0,\\
m=u-u_{xx},\qquad n=v-v_{xx}.
\end{cases}
\end{equation}
When $\rho\equiv0$, this system reduces to the GX system \eqref{GX}.
Related three-component integrable systems, their local theory, and
their degenerate forms have been studied in
\cite{li2018Degenerate,Li2023,Li2020local}.

We prove finite-time blow-up under the assumptions
$u_0\in W^{p,q}(\R)$ and $n_0\in L^{r'}(\R)$, with the exponents related
as in Theorem~\ref{TGX}. The proof uses the balance law for $nu_x$.
In conservative form, the term containing $nvu_x^2$ cancels, which allows
us to estimate convolutions of $|u|$, $|u_x|$, $|n|$, and $|nu_x|$
without an a priori uniform bound on $u_x$. For the three-component
system, the corresponding quantity is $g=nu_x+\rho^2-\bar\rho^2$,
where $\bar\rho\in[0,1]$. The density terms combine with the equation
for $nu_x$ to give the same cancellation. The assumptions on $p,q,r$
also give $n_0u_{0,x}\in L^1$ by Sobolev embedding and H\"older's inequality.

The GX blow-up result in \cite[Theorem~1.3]{liu2026} assumes
$u_0\in W^{1,1}\cap H^s$ and $v_0\in H^s$ with $s>\frac52$.
Theorem~\ref{TGX} allows $s\geq2$ under the stated integrability assumptions.
The case $p=q=1$ gives the same $W^{1,1}$--$L^\infty$ pairing, while
the other choices of $p,q,r$ give further sufficient conditions for blow-up.
The theorem also gives an explicit upper bound on the lifespan.

Following the approach of \cite{Meng2024828}, we use these estimates
to control $v$ along the characteristic starting at a point $x_0$ with
$v_0(x_0)>0$. On an explicitly determined time interval,
$v(t,X(t,x_0))\geq v_0(x_0)/2$. The slope equation then gives a Riccati
inequality, which forces blow-up if $u_{0,x}(x_0)$ is sufficiently negative.

We also prove local well-posedness for the three-component
Degasperis--Procesi system in $H^{s-1}\times H^s\times H^s$ for
$s>\frac32$, using the estimates of \cite{himonas2016novikov}.
Momentum estimates and dyadic transport estimates yield a blow-up
criterion in $B^1_{2,1}\times B^2_{2,1}\times B^2_{2,1}$.

Section 2 collects the definitions and estimates used below.
Section 3 proves local well-posedness and the continuation criteria.
Section 4 proves the finite-time blow-up results, and Section 5 gives
examples of initial data satisfying their hypotheses.

\subsection{Main results}
We first state the blow-up criterion for the GX system.

\begin{theo}\label{TGX}
Let $s\geq2$, $1\leq p\leq2$, and $1\leq q<\infty$.
Choose $r$ such that
\[
r\in
\begin{cases}
\displaystyle\frac{q}{1-q(p-1)},&p-1<1/q,\\[5pt]
[q,\infty),&p-1=1/q,\quad q>1,\\[3pt]
\infty,&p-1>1/q\quad\text{or}\quad(p,q)=(2,1).
\end{cases}
\]
Denote by $r'$ the H\"older conjugate of $r$.
Assume that $(u_0,v_0)\in H^s(\R)^2$, $u_0\in W^{p,q}(\R)$, and
$n_0:=v_0-v_{0,xx}$ belongs to $L^{r'}(\R)$.
Let $(u,v)$ be the corresponding $H^s$ strong solution to the Cauchy
problem for \eqref{GX}, and denote its maximal lifespan by $T^*$.
Set
\[
K=\left[1+\bigl(\|u_0\|_{L^q}+\|u_{0,x}\|_{L^q}\bigr)^2
+\|n_0\|_{L^{r'}}^2+\tfrac12\|n_0u_{0,x}\|_{L^1}\right]^{1/2}.
\]
Suppose that $v_0(x_0)>0$ at some $x_0\in\R$, and define
\[
\tau=\min\left\{\frac{1}{20K^3},
\frac{v_0(x_0)}{48K^3}\right\},
\qquad
 b=\sqrt{4K^4+32v_0(x_0)K^3}.
\]
If
\[
-u_{0,x}(x_0)>
\frac{2K^2+b\cothop(b\tau/2)}{v_0(x_0)},
\]
then $(u,v)$ blows up in finite time. More precisely, with
\[
\lambda_\pm=\frac{2K^2\pm b}{v_0(x_0)},
\]
its maximal lifespan satisfies
\[
T^*\leq\frac{1}{b}\log\left(
\frac{u_{0,x}(x_0)+\lambda_-}{u_{0,x}(x_0)+\lambda_+}\right)<\tau.
\]
\end{theo}

For $1<p<2$, $W^{p,q}$ denotes the nonhomogeneous
Sobolev--Slobodeckij space; thus $u\in W^{p,q}$ means
$u\in W^{1,q}$ and
\[
\int_{\R}\int_{\R}
\frac{|u_x(x)-u_x(y)|^q}{|x-y|^{1+(p-1)q}}\,\mathrm dx\,\mathrm dy<\infty.
\] The stated choices of $r$ give
$W^{p-1,q}(\R)\hookrightarrow L^r(\R)$; hence
$n_0u_{0,x}\in L^1(\R)$ by H\"older's inequality.

For $p=q=1$, $r=1$ and $r'=\infty$, so the assumptions use the
$W^{1,1}$--$L^\infty$ pairing. No sign condition is imposed on $m_0$ or
$n_0$; only $v_0(x_0)>0$ is required. The full $H^s$ norms do not enter
$K$. Example~\ref{examplepaired} gives smooth data for which $K$ stays
bounded while $-u_{0,x}(x_0)$ tends to infinity.

For the three-component system \eqref{(1.1)} with density background
$1$, we replace $nu_x$ by $nu_x+\rho^2-1$. Writing the density
perturbation as $\rho-1$, the Cauchy problem becomes
\begin{equation}\label{(1.2)}\tag{1.3}
\begin{cases}
(\rho-1)_t+((\rho-1)uv)_x+(uv)_x=0,\\
m_t+uvm_x+3vu_xm+\rho^2u=0,\\
n_t+uvn_x+3uv_xn-\rho^2v=0,\\
m=u-u_{xx},\quad n=v-v_{xx},\\
(\rho-1,u,v)|_{t=0}=(\rho_0-1,u_0,v_0).
\end{cases}
\end{equation}
Using the Green function $P(x)=\frac12e^{-|x|}$ and the identity
$D^{-2}f=P*f$, where $D^s=(1-\partial_x^2)^{s/2}$, we rewrite
\eqref{(1.2)} in the nonlocal form
\begin{equation}\label{(1.3)}\tag{1.4}
\begin{cases}
(\rho-1)_t+((\rho-1)uv)_x+(uv)_x=0,\\
u_t+uvu_x+P*(3uvu_x+2uv_xu_{xx}+2u_x^2v_x+uv_{xx}u_x+\rho^2u)=0,\\
v_t+uvv_x+P*(3uvv_x+2vu_xv_{xx}+2v_x^2u_x+vu_{xx}v_x-\rho^2v)=0,\\
(\rho-1,u,v)|_{t=0}=(\rho_0-1,u_0,v_0).
\end{cases}
\end{equation}

On $\T$, convolution with $P$ is understood using the periodic Green
kernel of $1-\partial_x^2$. The quantitative blow-up criteria below are
stated on $\R$.

Write $\mathcal H^s=H^{s-1}\times H^s\times H^s$ and
$\|(\eta,u,v)\|_{\mathcal H^s}
=\|\eta\|_{H^{s-1}}+\|u\|_{H^s}+\|v\|_{H^s}$.
We then have the following local well-posedness result.

\begin{theo}\label{3DP}
Let $s>\frac32$ and
$U_0=(\rho_0-1,u_0,v_0)\in\mathcal H^s$ on $\R$ or $\T$.
There is a unique strong solution
$U=(\rho-1,u,v)\in C([0,T];\mathcal H^s)
\cap C^1([0,T];\mathcal H^{s-1})$
to \eqref{(1.3)}. The data-to-solution map is continuous in these
spaces on a common local existence interval. If
$Y_0=1+\|U_0\|_{\mathcal H^s}$, then, for a constant $C_s>0$,
\[
1+\|U(t)\|_{\mathcal H^s}\leq\sqrt2\,Y_0,\qquad
0\leq t\leq T:=\frac{1}{4C_sY_0^2}.
\]
\end{theo}

The same local result holds with $\rho-1$ replaced by
$\rho-\bar\rho$ for any fixed $\bar\rho\in[0,1]$.
Also, $B^2_{2,1}\hookrightarrow H^2$ and
$B^1_{2,1}\hookrightarrow H^1$, so the following initial data lie in
$\mathcal H^2$.

\begin{theo}\label{Besovcriterion}
Let $(\rho_0-1,u_0,v_0)\in B^1_{2,1}\times B^2_{2,1}\times B^2_{2,1}$,
and let $T^*$ be the maximal lifespan in this class of the solution to
\eqref{(1.3)}. If $T^*<\infty$, then
\[
\int_0^{T^*}\|u(t)\|_{W^{1,\infty}}\|v(t)\|_{W^{1,\infty}}\dd t=\infty.
\]
\end{theo}

\begin{rema}
The same criterion holds in Sobolev spaces. Let
$(\rho_0-1,u_0,v_0)\in H^{s-1}\times H^s\times H^s$ with $s\geq2$,
and let $T^*$ be the maximal lifespan of the corresponding strong
solution to \eqref{(1.3)}. If $T^*<\infty$, then
\[
\int_0^{T^*}\|u(t)\|_{W^{1,\infty}}\|v(t)\|_{W^{1,\infty}}\,\mathrm dt=\infty.
\]
\end{rema}

We first state a blow-up condition for $u_0\in W^{1,1}(\R)$.

\begin{theo}[Finite-time blow-up for $u_0\in W^{1,1}\cap H^s$]
\label{Theorem 1.4}
Let $s>\frac52$, and assume that
\[
\rho_0-1\in W^{1,1}(\R)\cap H^{s-1}(\R),\qquad
u_0\in W^{1,1}(\R)\cap H^s(\R),\qquad v_0\in H^s(\R).
\]
Let $(\rho,u,v)$ be the corresponding strong solution to \eqref{(1.3)},
and denote its maximal lifespan by $T^*>0$. Set
\[
Q_0:=1+\|n_0\|_{L^\infty}
+\frac12\|\rho_0-1\|_{W^{1,1}}+\|u_0\|_{W^{1,1}}.
\]
Suppose that $v_0(x_0)>0$ at some $x_0\in\R$, and choose
\[
T_0\in\left(0,\frac{3v_0(x_0)}{184Q_0^3}\right).
\]
Define $b_1=Q_0^4/v_0(x_0)+6Q_0^3$.
If
\begin{equation}\label{(1.4)}\tag{1.5}
u_{0,x}(x_0)\leq
2\,\frac{1+e^{\sqrt{v_0(x_0)b_1}\,T_0}}
{1-e^{\sqrt{v_0(x_0)b_1}\,T_0}}
\sqrt{\frac{b_1}{v_0(x_0)}},
\end{equation}
then $(\rho,u,v)$ blows up in finite time and $T^*\leq T_0$.
\end{theo}

The following criterion treats the three-component system with a
constant density background $\bar\rho\in[0,1]$ and includes the GX
system when $\bar\rho=\rho_0=0$.

\begin{theo}[Finite-time blow-up for the three-component DP system]
\label{TDPpaired}
Let $s,p,q,r,r'$ be as in Theorem~\ref{TGX}, and fix $\bar\rho\in[0,1]$.
Assume that
\[
(\rho_0-\bar\rho,u_0,v_0)\in H^{s-1}\times H^s\times H^s,\qquad
u_0\in W^{p,q}(\R),
\]
and that $n_0\in L^{r'}(\R)$ and $\rho_0^2-\bar\rho^2\in L^1(\R)$.
Let $(\rho,u,v)$ be the maximal strong solution to the Cauchy problem
for \eqref{(1.1)} with these initial data, satisfying
$(\rho-\bar\rho,u,v)\in C([0,T^*);\mathcal H^s)$.
Set
\[
g=nu_x+\rho^2-\bar\rho^2,\qquad
 g_0=n_0u_{0,x}+\rho_0^2-\bar\rho^2,
\]
\[
K_\rho=\left[1+\bigl(\|u_0\|_{L^q}+\|u_{0,x}\|_{L^q}\bigr)^2
+\|n_0\|_{L^{r'}}^2+\tfrac12\|g_0\|_{L^1}\right]^{1/2}.
\]
Suppose that $v_0(x_0)>0$ at some $x_0\in\R$, and define
\[
\tau_\rho=\min\left\{\frac{1}{20K_\rho^3},
\frac{v_0(x_0)}{48K_\rho^3}\right\},\qquad
b_\rho=\sqrt{4K_\rho^4+32v_0(x_0)K_\rho^3}.
\]
If
\[
-u_{0,x}(x_0)>
\frac{2K_\rho^2+b_\rho\cothop(b_\rho\tau_\rho/2)}{v_0(x_0)},
\]
then $(\rho,u,v)$ blows up in finite time. More precisely, with
$\lambda_\pm=(2K_\rho^2\pm b_\rho)/v_0(x_0)$,
\[
T^*\leq\frac1{b_\rho}\log\left(
\frac{u_{0,x}(x_0)+\lambda_-}{u_{0,x}(x_0)+\lambda_+}\right)
<\tau_\rho.
\]
\end{theo}

\begin{rema}
When $\bar\rho=0$ and $\rho_0\equiv0$, the transport equation for $\rho$
implies that $\rho\equiv0$ throughout the lifespan. Hence $g_0=n_0u_{0,x}$,
and the criterion reduces exactly to Theorem~\ref{TGX}.
The choice $\bar\rho=1$ recovers the unit-background case.
The balance law for $g$ is given in \eqref{eq414}.
\end{rema}

\section{Preliminaries}
We recall the Littlewood--Paley decomposition, the nonhomogeneous
Besov spaces, and the estimates used below. On $\T$, we use the
corresponding periodic spaces.

\begin{prop}[\cite{BCD}]
Let $\mathcal C$ be an annulus and $\mathcal B$ a ball. There is a
constant $C>0$ such that, for any $\lambda>0$, $k\in\mathbb N$,
$1\leq p\leq q\leq\infty$, and $u\in L^p(\R^d)$,
\[
\supp\widehat u\subset\lambda\mathcal B
\ \Longrightarrow\
\|\nabla^ku\|_{L^q}:=\sup_{|\alpha|=k}\|\partial^\alpha u\|_{L^q}
\leq C^{k+1}\lambda^{k+d(1/p-1/q)}\|u\|_{L^p},
\]
\[
\supp\widehat u\subset\lambda\mathcal C
\ \Longrightarrow\
C^{-k-1}\lambda^k\|u\|_{L^p}
\leq\|\nabla^ku\|_{L^p}
\leq C^{k+1}\lambda^k\|u\|_{L^p}.
\]
\end{prop}

\begin{prop}[\cite{BCD}]
Let $\mathcal B=\{\xi\in\R^d:|\xi|\leq4/3\}$ and
$\mathcal C=\{\xi\in\R^d:3/4\leq|\xi|\leq8/3\}$.
There exist two smooth radial functions $\chi,\varphi$, valued in $[0,1]$
and supported in $\mathcal B,\mathcal C$, respectively, such that
\[
\chi(\xi)+\sum_{j\geq0}\varphi(2^{-j}\xi)=1
\quad\text{for all }\xi\in\R^d,
\]
\[
j\geq1\quad\Longrightarrow\quad
\supp\chi\cap\supp\varphi(2^{-j}\cdot)=\varnothing.
\]
For $u\in\mathcal S'$, define
\[
\Delta_ju:=0\quad(j\leq-2),\qquad
\Delta_{-1}u:=\chi(D)u=\mathcal F^{-1}(\chi\mathcal Fu),
\]
\[
\Delta_ju:=\varphi(2^{-j}D)u
=\mathcal F^{-1}\bigl(\varphi(2^{-j}\cdot)\mathcal Fu\bigr)
\quad(j\geq0),\qquad S_ju:=\sum_{j'\leq j-1}\Delta_{j'}u.
\]
Then the Littlewood--Paley decomposition is
\[
u=\sum_{j\in\mathbb Z}\Delta_ju\qquad\text{in }\mathcal S'.
\]
\end{prop}

\begin{defi}[\cite{BCD}]
Let $s\in\R$ and $(p,r)\in[1,\infty]^2$.
The nonhomogeneous Besov space $B^s_{p,r}$ consists of all
$u\in\mathcal S'$ such that
\[
\|u\|_{B^s_{p,r}}
:=\left\|(2^{js}\|\Delta_ju\|_{L^p})_{j\in\mathbb Z}
\right\|_{\ell^r(\mathbb Z)}<\infty.
\]
\end{defi}

\begin{defi}[\cite{BCD}]
For $u,v\in\mathcal S'$, whenever the product is defined,
Bony's decomposition reads
\[
uv=T_uv+T_vu+R(u,v),
\]
where
\[
T_uv=\sum_j S_{j-1}u\,\Delta_jv,\qquad
R(u,v)=\sum_{|j-k|\leq1}\Delta_ku\,\Delta_jv.
\]
\end{defi}

\begin{lem}[\cite{BCD}]
For $s\in\R$,
$\|T_uv\|_{B^s_{p,r}}\leq C\|u\|_\infty\|v\|_{B^s_{p,r}}$.
If $t<0$ and $1/r=1/r_1+1/r_2\leq1$, then
\[
\|T_uv\|_{B^{s+t}_{p,r}}
\leq C\|u\|_{B^t_{\infty,r_1}}\|v\|_{B^s_{p,r_2}}.
\]
If $1/p=1/p_1+1/p_2\leq1$, $1/r=1/r_1+1/r_2\leq1$, and
$s_1+s_2>0$, then
\[
\|R(u,v)\|_{B^{s_1+s_2}_{p,r}}
\leq C\|u\|_{B^{s_1}_{p_1,r_1}}\|v\|_{B^{s_2}_{p_2,r_2}}.
\]
When $s_1+s_2=0$ and $r=1$, the target space is $B^0_{p,\infty}$.
\end{lem}

We shall use the following commutator and product estimates.

\begin{lem}[\cite{kato1988}]
If $s>0$, then there is a constant $c_s>0$ such that
\[
\|[D^s,f]g\|_{L^2}
\leq c_s\bigl(\|D^sf\|_{L^2}\|g\|_\infty
+\|f_x\|_\infty\|D^{s-1}g\|_{L^2}\bigr).
\]
\end{lem}

\begin{lem}[\cite{himonas2016novikov,himonas2010}]
If $\sigma>\frac12$, then there is a constant $c_\sigma>0$ such that
\[
\|fg\|_{H^{\sigma-1}}\leq c_\sigma
\|f\|_{H^\sigma}\|g\|_{H^{\sigma-1}}.
\]
\end{lem}

\begin{lem}[\cite{himonas2009}]
Let $g$ satisfy $\|\partial_xg\|_\infty<\infty$.
For a Friedrichs mollifier $J_\epsilon$, there is a constant $C>0$,
independent of $\epsilon$, such that, for every $f\in L^2$,
\[
\|[J_\epsilon,g]\partial_xf\|_{L^2}
\leq C\|\partial_xg\|_\infty\|f\|_{L^2}.
\]
\end{lem}

\begin{lem}[\cite{taylor2003}]
If $\sigma+1\geq0$, then, for $f\in H^\rho$ and $g\in H^\sigma$,
\[
\|[D^\sigma\partial_x,f]g\|_{L^2}
\lesssim\|f\|_{H^\rho}\|g\|_{H^\sigma},
\]
provided that $\rho>\frac32$ and $\sigma+1\leq\rho$.
\end{lem}

\begin{lem}[\cite{chen2016}]
Let $a,b>0$ and $f\in C^1([0,T))$ satisfy
$f'\leq-af^2+b$ and $f(0)<-\sqrt{b/a}$.
Then Riccati comparison gives
\[
T\leq t_R:=\frac1{2\sqrt{ab}}\log\left(
\frac{f(0)-\sqrt{b/a}}{f(0)+\sqrt{b/a}}\right).
\]
\end{lem}

We also recall the transport estimates used below \cite{BCD}.
For $f_t+w f_x=S$ and $s>\frac32$,
\[
D^+\|f\|_{H^{s-1}}
\leq C_s\|w\|_{H^s}\|f\|_{H^{s-1}}+\|S\|_{H^{s-1}}.
\]
For $\sigma=0,1$ and $w\in H^2$,
\[
D^+\|f\|_{B^\sigma_{2,1}}
\leq C\|w\|_{H^2}\|f\|_{B^\sigma_{2,1}}
+\|S\|_{B^\sigma_{2,1}}.
\]
Here $D^+$ denotes the upper Dini derivative; for nonsmooth norms we use
the corresponding integrated inequalities.
The coefficients are admissible because
$w_x\in H^1\hookrightarrow B^{1/2}_{2,1}\cap L^\infty$.
Bony's decomposition also gives
\[
\|af\|_{B^0_{2,1}}\leq C\|a\|_{H^1}\|f\|_{B^0_{2,1}},
\qquad
\|fg\|_{B^1_{2,1}}\leq
C\bigl(\|f\|_\infty\|g\|_{B^1_{2,1}}
+\|g\|_\infty\|f\|_{B^1_{2,1}}\bigr).
\]
For the first estimate one may use
$H^1\hookrightarrow B^{1/2}_{\infty,2}$ in the paraproduct and remainder
terms. Constant parts of coefficients are estimated separately.

\section{Local well-posedness}
We write $\eta=\rho-\bar\rho$, where $\bar\rho\in[0,1]$ is fixed,
and use the mixed space
$\mathcal H^s=H^{s-1}\times H^s\times H^s$ introduced above.
For the unit-background problem \eqref{(1.3)}, $\bar\rho=1$.
Since the transport terms $uv\eta_x$, $vu u_x$, and $uvv_x$ lose one
spatial derivative, \eqref{(1.3)} cannot be treated directly as a system of
ODEs in $\mathcal H^s$. We therefore use the Friedrichs mollifier $J_\epsilon$.
We first fix a real, even Schwartz function $j$ such that
$0\leq\widehat j(\xi)\leq1$, $\widehat j(\xi)=1$ for $|\xi|\leq1$,
and $\widehat j$ has compact support. Put
\begin{equation}\tag{3.1}\label{eq31}
j_\epsilon(x)=\epsilon^{-1}j(x/\epsilon),\qquad
J_\epsilon f=j_\epsilon*f.
\end{equation}
For convenience, write the nonlocal terms as
\[
\begin{aligned}
F(\eta,u,v)&=D^{-2}\{3uvu_x+(uu_xv_x)_x+u_x^2v_x+uv_xu_{xx}
 +(\bar\rho+\eta)^2u\},\\
G(\eta,u,v)&=D^{-2}\{3uvv_x+(vv_xu_x)_x+v_x^2u_x+vu_xv_{xx}
 -(\bar\rho+\eta)^2v\}.
\end{aligned}
\]
Set $U_\epsilon=(\eta_\epsilon,u_\epsilon,v_\epsilon)$ and
$\widetilde U_\epsilon=J_\epsilon U_\epsilon
=(\widetilde\eta,\widetilde u,\widetilde v)$.
The mollified initial value problem is
\begin{equation}\tag{3.2}\label{eq32}
\begin{cases}
\eta_{\epsilon,t}+J_\epsilon\bigl[
 \widetilde u\widetilde v\widetilde\eta_x
 +(\bar\rho+\widetilde\eta)(\widetilde u_x\widetilde v
 +\widetilde u\widetilde v_x)\bigr]=0,\\
u_{\epsilon,t}+J_\epsilon(\widetilde u\widetilde v\widetilde u_x)
 +J_\epsilon F(\widetilde U_\epsilon)=0,\\
v_{\epsilon,t}+J_\epsilon(\widetilde u\widetilde v\widetilde v_x)
 +J_\epsilon G(\widetilde U_\epsilon)=0,\\
U_\epsilon(0)=J_\epsilon U_0.
\end{cases}
\end{equation}
For fixed $\epsilon>0$, its right-hand side is locally Lipschitz in
$\mathcal H^s$. The ODE theorem in Banach spaces \cite{dieudonne2011}
therefore gives a unique local solution. We next obtain bounds independent
of $\epsilon$.

\textbf{Uniform energy estimates.}
The transport estimate at index $s-1$, self-adjointness of
$J_\epsilon$, and the algebra property of $H^{s-1}$ yield
\begin{equation}\tag{3.3}\label{eq33}
D^+\|\eta_\epsilon\|_{H^{s-1}}
\lesssim(1+\|\eta_\epsilon\|_{H^{s-1}})
\|u_\epsilon\|_{H^s}\|v_\epsilon\|_{H^s}.
\end{equation}
Applying $D^s$ to the second equation of \eqref{eq32}, taking its
$L^2$ inner product with $D^su_\epsilon$, and integrating over $\R$, we obtain
\begin{equation}\tag{3.4}\label{eq34}
\begin{aligned}
\frac12\frac{\mathrm d}{\mathrm dt}\|u_\epsilon\|_{H^s}^2
={}&-\int_\R D^s(\widetilde u\widetilde v\widetilde u_x)
                  D^s\widetilde u\,\mathrm dx\\
&-\int_\R D^s F(\widetilde U_\epsilon)
                  D^s\widetilde u\,\mathrm dx.
\end{aligned}
\end{equation}
Here we have used the self-adjointness of $J_\epsilon$ and its
commutation with $D^s$. For the transport term, we have
\[
\begin{aligned}
\int_\R D^s(\widetilde u\widetilde v\widetilde u_x)D^s\widetilde u\,\mathrm dx
={}&\int_\R[D^s,\widetilde u\widetilde v]\widetilde u_x
             D^s\widetilde u\,\mathrm dx\\
&+\int_\R\widetilde u\widetilde v D^s\widetilde u_x
             D^s\widetilde u\,\mathrm dx.
\end{aligned}
\]
By Cauchy--Schwarz and Lemma 2.6,
\[
\begin{aligned}
\left|\int_\R[D^s,\widetilde u\widetilde v]\widetilde u_x
 D^s\widetilde u\,\mathrm dx\right|
&\leq\|[D^s,\widetilde u\widetilde v]\widetilde u_x\|_2
       \|\widetilde u\|_{H^s}\\
&\lesssim\bigl(\|\widetilde u\widetilde v\|_{H^s}
                     \|\widetilde u_x\|_\infty
 +\|(\widetilde u\widetilde v)_x\|_\infty
                     \|\widetilde u_x\|_{H^{s-1}}\bigr)
                     \|\widetilde u\|_{H^s}.
\end{aligned}
\]
Combining the Sobolev algebra property and the embedding
$H^s\hookrightarrow W^{1,\infty}$ for $s>3/2$, we find
\begin{equation}\tag{3.5}\label{eq35}
\begin{aligned}
\left|\int_\R[D^s,\widetilde u\widetilde v]\widetilde u_x
 D^s\widetilde u\,\mathrm dx\right|
&\lesssim\bigl(\|\widetilde v\|_{H^s}\|\widetilde u\|_{H^s}
                     \|\widetilde u\|_{W^{1,\infty}}
 +\|\widetilde u\widetilde v\|_{W^{1,\infty}}
                     \|\widetilde u\|_{H^s}\bigr)\|\widetilde u\|_{H^s}\\
&\lesssim\|v_\epsilon\|_{H^s}\|u_\epsilon\|_{H^s}^3.
\end{aligned}
\end{equation}
The remaining transport integral is estimated by integration by parts:
\[
\int_\R\widetilde u\widetilde vD^s\widetilde u_xD^s\widetilde u\,\mathrm dx
=-\frac12\int_\R(\widetilde u\widetilde v)_x|D^s\widetilde u|^2\,\mathrm dx,
\]
so its absolute value has the same upper bound as in \eqref{eq35}.
We next consider the nonlocal term in \eqref{eq34}. By the definition of $F$,
\begin{equation}\tag{3.6}\label{eq36}
\begin{aligned}
\left|\int_\R D^sF(\widetilde U_\epsilon)D^s\widetilde u\,\mathrm dx\right|
\lesssim\bigl(&\|\widetilde v\widetilde u\widetilde u_x\|_{H^{s-2}}
+\|(\widetilde u\widetilde u_x\widetilde v_x)_x\|_{H^{s-2}}\\
&+\|\widetilde v_x\widetilde u_x^2\|_{H^{s-2}}
+\|\widetilde u\widetilde v_x\widetilde u_{xx}\|_{H^{s-2}}\\
&+\|(\bar\rho+\widetilde\eta)^2\widetilde u\|_{H^{s-2}}\bigr)
\|\widetilde u\|_{H^s}.
\end{aligned}
\end{equation}
Using the algebra property of $H^{s-1}$ and Lemma 2.7, we obtain
\begin{equation}\tag{3.7}\label{eq37}
\begin{aligned}
\|\widetilde v\widetilde u\widetilde u_x\|_{H^{s-2}}
&\lesssim\|\widetilde v\|_{H^{s-1}}\|\widetilde u\|_{H^{s-1}}
       \|\widetilde u_x\|_{H^{s-1}}
\lesssim\|v_\epsilon\|_{H^s}\|u_\epsilon\|_{H^s}^2,\\
\|(\widetilde u\widetilde u_x\widetilde v_x)_x\|_{H^{s-2}}
&\lesssim\|\widetilde u\widetilde u_x\widetilde v_x\|_{H^{s-1}}
\lesssim\|v_\epsilon\|_{H^s}\|u_\epsilon\|_{H^s}^2,\\
\|\widetilde v_x\widetilde u_x^2\|_{H^{s-2}}
&\lesssim\|\widetilde v_x\widetilde u_x^2\|_{H^{s-1}}
\lesssim\|v_\epsilon\|_{H^s}\|u_\epsilon\|_{H^s}^2,\\
\|(\bar\rho+\widetilde\eta)^2\widetilde u\|_{H^{s-2}}
&\lesssim(1+\|\eta_\epsilon\|_{H^{s-1}})^2\|u_\epsilon\|_{H^s}.
\end{aligned}
\end{equation}
For the last estimate, expand
$(\bar\rho+\widetilde\eta)^2\widetilde u
=\bar\rho^2\widetilde u+2\bar\rho\widetilde\eta\widetilde u
+\widetilde\eta^2\widetilde u$, estimating the constant part separately.
For the term containing $\widetilde u_{xx}$ in \eqref{eq36}, the algebra
estimate in $H^{s-2}$ would require $s>5/2$. Lemma 2.7 instead gives,
for every $s>3/2$,
\begin{equation}\tag{3.8}\label{eq38}
\begin{aligned}
\|\widetilde u\widetilde v_x\widetilde u_{xx}\|_{H^{s-2}}
&\lesssim\|\widetilde u\widetilde v_x\|_{H^{s-1}}
             \|\widetilde u_{xx}\|_{H^{s-2}}\\
&\lesssim\|\widetilde u\|_{H^{s-1}}\|\widetilde v_x\|_{H^{s-1}}
             \|\widetilde u\|_{H^s}
\lesssim\|v_\epsilon\|_{H^s}\|u_\epsilon\|_{H^s}^2.
\end{aligned}
\end{equation}
Consequently,
\begin{equation}\tag{3.9}\label{eq39}
\left|\int_\R D^sF(\widetilde U_\epsilon)D^s\widetilde u\,\mathrm dx\right|
\lesssim\|v_\epsilon\|_{H^s}\|u_\epsilon\|_{H^s}^3
+(1+\|\eta_\epsilon\|_{H^{s-1}})^2\|u_\epsilon\|_{H^s}^2.
\end{equation}
Together with the transport estimates, this implies
\begin{equation}\tag{3.10}\label{eq310}
D^+\|u_\epsilon\|_{H^s}
\lesssim\|v_\epsilon\|_{H^s}\|u_\epsilon\|_{H^s}^2
+(1+\|\eta_\epsilon\|_{H^{s-1}})^2\|u_\epsilon\|_{H^s}.
\end{equation}
The analogous inequality for $v_\epsilon$ is
\begin{equation}\tag{3.11}\label{eq311}
D^+\|v_\epsilon\|_{H^s}
\lesssim\|u_\epsilon\|_{H^s}\|v_\epsilon\|_{H^s}^2
+(1+\|\eta_\epsilon\|_{H^{s-1}})^2\|v_\epsilon\|_{H^s}.
\end{equation}
Set $Y_\epsilon=1+\|\eta_\epsilon\|_{H^{s-1}}
+\|u_\epsilon\|_{H^s}+\|v_\epsilon\|_{H^s}$ and
$Y_0=1+\|U_0\|_{\mathcal H^s}$.
Combining \eqref{eq33}, \eqref{eq310}, and \eqref{eq311}, we obtain
\begin{equation}\tag{3.12}\label{eq312}
D^+Y_\epsilon\leq C_sY_\epsilon^3.
\end{equation}
Integrating this differential inequality gives
\[
Y_\epsilon(t)\leq\frac{Y_0}{\sqrt{1-2C_sY_0^2t}}.
\]
Thus the mollified solutions exist on a common time interval $[0,T]$, with
\begin{equation}\tag{3.13}\label{eq313}
T=\frac1{4C_sY_0^2},
\end{equation}
and satisfy
\begin{equation}\tag{3.14}\label{eq314}
1+\|U_\epsilon(t)\|_{\mathcal H^s}\leq\sqrt2\,Y_0,
\qquad 0\leq t\leq T.
\end{equation}
The ODE continuation theorem therefore extends every mollified
solution to $[0,T]$.
The equations also give a uniform bound for the time derivatives.
Indeed, the density equation and Lemma 2.7 imply
\[
\begin{aligned}
\|\eta_{\epsilon,t}\|_{H^{s-2}}
&\lesssim\|\widetilde u\widetilde v\widetilde\eta_x\|_{H^{s-2}}
+\|(\bar\rho+\widetilde\eta)(\widetilde u\widetilde v)_x\|_{H^{s-2}}\\
&\lesssim(1+\|\eta_\epsilon\|_{H^{s-1}})
\|u_\epsilon\|_{H^s}\|v_\epsilon\|_{H^s}.
\end{aligned}
\]
For the two velocity components, the estimates above yield
\[
\begin{aligned}
\|u_{\epsilon,t}\|_{H^{s-1}}
&\lesssim\|\widetilde u\widetilde v\widetilde u_x\|_{H^{s-1}}
 +\|F(\widetilde U_\epsilon)\|_{H^s}\lesssim Y_0^3,\\
\|v_{\epsilon,t}\|_{H^{s-1}}
&\lesssim\|\widetilde u\widetilde v\widetilde v_x\|_{H^{s-1}}
 +\|G(\widetilde U_\epsilon)\|_{H^s}\lesssim Y_0^3.
\end{aligned}
\]

We now show that a subsequence converges to a solution of \eqref{(1.3)}.
The limiting solution will satisfy
\begin{equation}\tag{3.15}\label{eq315}
1+\|U(t)\|_{\mathcal H^s}\leq\sqrt2\,Y_0,
\qquad 0\leq t\leq T,
\end{equation}
and
\begin{equation}\tag{3.16}\label{eq316}
\|\partial_tU_\epsilon\|_{\mathcal H^{s-1}},\
\|\partial_tU\|_{\mathcal H^{s-1}}\leq C_sY_0^3,
\qquad\mathcal H^{s-1}=H^{s-2}\times H^{s-1}\times H^{s-1}.
\end{equation}

\textbf{Compactness.}
Let $I=[0,T]$. By \eqref{eq314}, the family $\{U_\epsilon\}$ is bounded in
$L^\infty(I;\mathcal H^s)$. The Banach--Alaoglu theorem gives a
subsequence $U_{\epsilon_j}$ and a limit $U$ such that
\begin{equation}\tag{3.17}\label{eq317}
U_{\epsilon_j}\rightharpoonup^*U
\quad\text{in }L^\infty(I;\mathcal H^s).
\end{equation}
Weak lower semicontinuity gives \eqref{eq315} initially for almost
 every time. It remains to prove that $U$ solves the equation and belongs
to $C(I;\mathcal H^s)$.

Let $\chi\in C_c^\infty(\R)$ and $0<\sigma<1$.
Multiplication by $\chi$ is compact from $H^a$ into $H^{a-\sigma}$ for
any fixed $a\in\R$. Moreover, interpolation between \eqref{eq314}
and the uniform time-derivative bound in \eqref{eq316} gives
\[
\|U_\epsilon(t_1)-U_\epsilon(t_2)\|_{\mathcal H^{s-\sigma}}
\leq C_{s,Y_0}|t_1-t_2|^\sigma.
\]
The Arzel\`a--Ascoli theorem therefore yields, after extracting a diagonal
subsequence,
\[
\chi U_{\epsilon_j}\longrightarrow\chi U
\quad\text{in }C(I;\mathcal H^{s-\sigma}),\qquad
U_{\epsilon_j}\longrightarrow U\quad\text{in }C(I;\mathcal H^{s-\sigma}_{\rm loc}).
\]

\textbf{Passage to the limit.}
Choose $0<\sigma<\min\{1,s-3/2\}$ and put $a=s-\sigma>3/2$.
The velocities then converge in $C(I;C^1_{\rm loc})$, and the density
perturbations converge in $C(I;H^{a-1}_{\rm loc})$, hence locally
uniformly in space and time. The following estimates justify the
passage to the limit without differentiating the density pointwise.
For the first velocity transport term, the algebra property gives
\begin{equation}\tag{3.18}\label{eq318}
\begin{aligned}
&\|J_\epsilon(\widetilde u\widetilde v\widetilde u_x)(t_1)
 -J_\epsilon(\widetilde u\widetilde v\widetilde u_x)(t_2)\|_{H^{a-1}}\\
&\quad\lesssim
\|\widetilde u(t_1)\widetilde v(t_1)\|_{H^a}
\|\widetilde u(t_1)-\widetilde u(t_2)\|_{H^a}\\
&\qquad+
\|\widetilde u(t_1)\widetilde v(t_1)
 -\widetilde u(t_2)\widetilde v(t_2)\|_{H^a}\|\widetilde u(t_2)\|_{H^a}\\
&\quad\leq C_{s,Y_0}|t_1-t_2|^\sigma.
\end{aligned}
\end{equation}
Similarly,
\begin{equation}\tag{3.19}\label{eq319}
\|J_\epsilon(\widetilde u\widetilde v\widetilde v_x)(t_1)
 -J_\epsilon(\widetilde u\widetilde v\widetilde v_x)(t_2)\|_{H^{a-1}}
\leq C_{s,Y_0}|t_1-t_2|^\sigma.
\end{equation}
The estimates \eqref{eq37}--\eqref{eq38}, with $s$ replaced by $a$,
also apply to differences. For example,
\[
\begin{aligned}
u v_xu_{xx}-\widehat u\widehat v_x\widehat u_{xx}
={}&(u-\widehat u)v_xu_{xx}
 +\widehat u(v_x-\widehat v_x)u_{xx}
 +\widehat u\widehat v_x(u_{xx}-\widehat u_{xx}),
\end{aligned}
\]
so Lemma 2.7 controls this difference in $H^{a-2}$ by
$C_R\|U-\widehat U\|_{\mathcal H^a}$ on a bounded ball of radius $R$.
The density term is handled by expanding
$(\bar\rho+\eta)^2u-(\bar\rho+\widehat\eta)^2\widehat u$.
Thus $F$ and $G$ are locally Lipschitz from $\mathcal H^a$ to $H^a$.
Consequently,
\begin{equation}\tag{3.20}\label{eq320}
\|J_\epsilon F(\widetilde U_\epsilon(t_1))
 -J_\epsilon F(\widetilde U_\epsilon(t_2))\|_{H^a}
\leq C_{s,Y_0}|t_1-t_2|^\sigma,
\end{equation}
and
\begin{equation}\tag{3.21}\label{eq321}
\|J_\epsilon G(\widetilde U_\epsilon(t_1))
 -J_\epsilon G(\widetilde U_\epsilon(t_2))\|_{H^a}
\leq C_{s,Y_0}|t_1-t_2|^\sigma.
\end{equation}
For the density equation, the conservative form gives
\begin{equation}\tag{3.22}\label{eq322}
\|J_\epsilon\partial_x[(\bar\rho+\widetilde\eta)
 \widetilde u\widetilde v](t_1)
 -J_\epsilon\partial_x[(\bar\rho+\widetilde\eta)
 \widetilde u\widetilde v](t_2)\|_{H^{a-2}}
\leq C_{s,Y_0}|t_1-t_2|^\sigma.
\end{equation}
These estimates give equicontinuity in time in the indicated Sobolev
spaces. To pass to the limit, we use local strong convergence and
Lemma 2.7 for products containing second derivatives.
For a nonlocal term $D^{-2}h_\epsilon$, test against a compactly supported
smooth function $\psi$ and write
$\langle D^{-2}h_\epsilon,\psi\rangle
=\langle h_\epsilon,D^{-2}\psi\rangle$.
The decay of $D^{-2}\psi$, together with the uniform global Sobolev bounds,
allows approximation by compactly supported test functions. The same
argument applies to $\partial_xD^{-2}$.
Passing to the limit in the time-integrated equations therefore gives
\[
\begin{cases}
\eta_t+\partial_x[(\bar\rho+\eta)uv]=0,\\
u_t+uvu_x+F(\eta,u,v)=0,\\
v_t+uvv_x+G(\eta,u,v)=0,\\
U(0)=U_0.
\end{cases}
\]
In particular, the limit has the time-derivative bound in \eqref{eq316}.
It follows that
\begin{equation}\tag{3.23}\label{eq323}
\|U(t_1)-U(t_2)\|_{\mathcal H^{s-1}}
\leq\sup_{t\in I}\|\partial_tU(t)\|_{\mathcal H^{s-1}}|t_1-t_2|
\leq C_sY_0^3|t_1-t_2|.
\end{equation}
Thus $U\in\operatorname{Lip}(I;\mathcal H^{s-1})$.

\textbf{Strong continuity in time.}
Fix $t\in I$ and a sequence $t_n\to t$. For the $u$ component,
\begin{equation}\tag{3.24}\label{eq324}
\begin{aligned}
\|u(t_n)-u(t)\|_{H^s}^2
&=\int_\R(1+\xi^2)^s|\widehat u(\xi,t_n)-\widehat u(\xi,t)|^2\,\mathrm d\xi\\
&=\|u(t_n)\|_{H^s}^2
 -2\operatorname{Re}\langle u(t_n),u(t)\rangle_{H^s}
 +\|u(t)\|_{H^s}^2.
\end{aligned}
\end{equation}
Let $\varphi\in H^s$ and $\delta>0$. Choose $\psi\in\mathcal S(\R)$ such
that $\|\varphi-\psi\|_{H^s}<\delta/(4M)$, where
$M=1+\sup_{t\in I}\|U(t)\|_{\mathcal H^s}$.
Then Cauchy--Schwarz yields
\begin{equation}\tag{3.25}\label{eq325}
\begin{aligned}
|\langle u(t_n)-u(t),\varphi\rangle_{H^s}|
&\leq|\langle u(t_n)-u(t),\varphi-\psi\rangle_{H^s}|
 +|\langle u(t_n)-u(t),\psi\rangle_{H^s}|\\
&\leq\|u(t_n)-u(t)\|_{H^s}\|\varphi-\psi\|_{H^s}
 +\|u(t_n)-u(t)\|_{H^{s-1}}\|\psi\|_{H^{s+1}}\\
&\leq\delta/2+C_sY_0^3|t_n-t|\|\psi\|_{H^{s+1}}.
\end{aligned}
\end{equation}
The last term is smaller than $\delta/2$ for all sufficiently large $n$.
Hence $u(t_n)\rightharpoonup u(t)$ in $H^s$.
The same reasoning applies to $v$ and to $\eta$ at index $s-1$.
Using this argument first on a full-measure set and the lower-space
continuity defines a weakly continuous representative at every time;
\eqref{eq315} then holds throughout $I$.
Consequently, \eqref{eq324} shows that strong continuity follows once the
corresponding Sobolev norms are continuous.

To avoid dividing by a possibly vanishing mollified norm, we work with
squared norms. It suffices to establish, for
$(h,\alpha)=(u,s),(v,s),(\eta,s-1)$, the uniform estimate
\begin{equation}\tag{3.26}\label{eq326}
\left|\|J_\epsilon h(t_1)\|_{H^\alpha}^2
 -\|J_\epsilon h(t_2)\|_{H^\alpha}^2\right|
\leq C_{s,Y_0}|t_1-t_2|.
\end{equation}
Indeed, $J_\epsilon h(t)\to h(t)$ in $H^\alpha$ for each fixed $t$,
so the same estimate holds for the limiting squared norm.
Applying $J_\epsilon$ to the equation for $u$ gives
$\partial_tJ_\epsilon u=-J_\epsilon(uvu_x)-J_\epsilon F$ and hence
\begin{equation}\tag{3.27}\label{eq327}
\begin{aligned}
\frac12\frac{\mathrm d}{\mathrm dt}\|J_\epsilon u\|_{H^s}^2
={}&-\int_\R D^sJ_\epsilon(uvu_x)D^sJ_\epsilon u\,\mathrm dx\\
&-\int_\R D^sJ_\epsilon F D^sJ_\epsilon u\,\mathrm dx.
\end{aligned}
\end{equation}
For the first integral,
\begin{equation}\tag{3.28}\label{eq328}
\begin{aligned}
\int_\R D^sJ_\epsilon(uvu_x)D^sJ_\epsilon u\,\mathrm dx
={}&\int_\R[D^s,uv]u_xD^sJ_\epsilon^2u\,\mathrm dx\\
&+\int_\R uvD^su_xD^sJ_\epsilon^2u\,\mathrm dx.
\end{aligned}
\end{equation}
Lemma 2.6, the algebra property, and Sobolev embedding give
\begin{equation}\tag{3.29}\label{eq329}
\begin{aligned}
\|[D^s,uv]u_x\|_2
&\lesssim\|uv\|_{H^s}\|u_x\|_\infty
 +\|(uv)_x\|_\infty\|u\|_{H^s}\\
&\lesssim\|v\|_{H^s}\|u\|_{H^s}^2.
\end{aligned}
\end{equation}
For the remaining term, commute $J_\epsilon$ with $uv$ and integrate by
parts to obtain
\begin{equation}\tag{3.30}\label{eq330}
\begin{aligned}
\int_\R uvD^su_xD^sJ_\epsilon^2u\,\mathrm dx
={}&\int_\R[J_\epsilon,uv]D^su_xD^sJ_\epsilon u\,\mathrm dx\\
&-\frac12\int_\R(uv)_x|D^sJ_\epsilon u|^2\,\mathrm dx.
\end{aligned}
\end{equation}
Using Lemma 2.8 with $f=D^su$ and $g=uv$, we find
\begin{equation}\tag{3.31}\label{eq331}
\begin{aligned}
\left|\int_\R uvD^su_xD^sJ_\epsilon^2u\,\mathrm dx\right|
&\lesssim\|(uv)_x\|_\infty\|u\|_{H^s}^2
 +\tfrac12\|(uv)_x\|_\infty\|u\|_{H^s}^2\\
&\lesssim\|v\|_{H^s}\|u\|_{H^s}^3.
\end{aligned}
\end{equation}
Together with \eqref{eq328}--\eqref{eq329}, this proves
\begin{equation}\tag{3.32}\label{eq332}
\left|\int_\R D^sJ_\epsilon(uvu_x)D^sJ_\epsilon u\,\mathrm dx\right|
\lesssim\|v\|_{H^s}\|u\|_{H^s}^3.
\end{equation}
For the nonlocal term in \eqref{eq327}, the preceding product estimates give
\[
\begin{aligned}
\left|\int_\R D^sJ_\epsilon F D^sJ_\epsilon u\,\mathrm dx\right|
&\leq\|F\|_{H^s}\|u\|_{H^s}\\
&\lesssim\|v\|_{H^s}\|u\|_{H^s}^3
+(1+\|\eta\|_{H^{s-1}})^2\|u\|_{H^s}^2.
\end{aligned}
\]
Thus the absolute value of the derivative of
$\|J_\epsilon u\|_{H^s}^2$ is bounded uniformly in $\epsilon$.
The same holds for $v$.
For the density perturbation, the transport commutator estimate at
index $s-1$ and the product bound in \eqref{eq33} give
\[
\left|\frac{\mathrm d}{\mathrm dt}\|J_\epsilon\eta\|_{H^{s-1}}^2\right|
\leq C_s(1+\|\eta\|_{H^{s-1}})\|\eta\|_{H^{s-1}}
\|u\|_{H^s}\|v\|_{H^s}\leq C_{s,Y_0}.
\]
Integrating in time proves \eqref{eq326}. Hence each component is strongly
continuous at its stated Sobolev index, and $U\in C(I;\mathcal H^s)$.
The equations then imply $U\in C^1(I;\mathcal H^{s-1})$.

\textbf{Uniqueness.}
We next prove uniqueness. Suppose that
$V=(a,b,c)\in C(I;\mathcal H^s)$ is another solution with the same initial
data. Set $\varpi=\eta-a$, $\phi=u-b$, and $\psi=v-c$.
Subtracting the equations gives
\begin{equation}\tag{3.33}\label{eq333}
\begin{cases}
\varpi_t=\partial_x\bigl[(\bar\rho+a)bc-(\bar\rho+\eta)uv\bigr],\\
\phi_t=-uv\phi_x-vb_x\phi-bb_x\psi-F(\eta,u,v)+F(a,b,c),\\
\psi_t=-uv\psi_x-uc_x\psi-cc_x\phi-G(\eta,u,v)+G(a,b,c),\\
(\varpi,\phi,\psi)|_{t=0}=0.
\end{cases}
\end{equation}
For $s>3/2$, we prove uniqueness in Lagrangian coordinates, where
the density is determined by the flow and its initial value. Let $X$ be the flow of $uv$ and put
$\mathsf u=u\circ X$, $\mathsf v=v\circ X$. The continuity equation gives
\begin{equation}\tag{3.34}\label{eq334}
X_t=\mathsf u\mathsf v,\qquad X(0,a)=a,\qquad
(\rho\circ X)X_a=\rho_0(a).
\end{equation}
Since $uv\in C(I;H^s)$ and $s>3/2$, this is an $H^s$ diffeomorphism
flow: $X-\mathrm{id}\in H^s$ and $\inf X_a>0$ on compact time intervals.
For $X=\mathrm{id}+\xi$ near the identity, define
\begin{equation}\tag{3.35}\label{eq335}
\partial_X=X_a^{-1}\partial_a,\qquad
L_X=1-\partial_X^2,\qquad
L_Xh=h-X_a^{-2}h_{aa}+X_a^{-3}X_{aa}h_a.
\end{equation}
The space $H^{s-1}$ is an algebra and, by Lemma 2.7, is a multiplier on
$H^{s-2}$. In particular, $L_X:H^s\to H^{s-2}$ is bounded and depends
smoothly on $\xi$. More explicitly,
\[
\begin{aligned}
\|(L_X-L_{\mathrm{id}})h\|_{H^{s-2}}
&\lesssim\|1-X_a^{-2}\|_{H^{s-1}}\|h_{aa}\|_{H^{s-2}}
 +\|X_a^{-3}h_a\|_{H^{s-1}}\|X_{aa}\|_{H^{s-2}}\\
&\leq C_s\|\xi\|_{H^s}\|h\|_{H^s}
\end{aligned}
\]
locally near $\xi=0$; constant parts of the coefficients are separated
in the multiplier estimates. Since
$L_{\mathrm{id}}=1-\partial_a^2:H^s\to H^{s-2}$ is an isomorphism,
a Neumann-series argument gives
\begin{equation}\tag{3.36}\label{eq336}
X\longmapsto L_X^{-1}\in\mathcal L(H^{s-2},H^s)
\quad\text{smoothly near the identity}.
\end{equation}
The Lagrangian density is explicitly determined by
\begin{equation}\tag{3.37}\label{eq337}
\mathsf r=\frac{\bar\rho+\eta_0}{X_a},\qquad
\mathsf r-\bar\rho
=\frac{\eta_0-\bar\rho(X_a-1)}{X_a}\in H^{s-1}.
\end{equation}
This expression depends smoothly on $(\xi,\eta_0)$ in
$H^s\times H^{s-1}$. Define the transformed nonlocal terms by
\begin{equation}\tag{3.38}\label{eq338}
\begin{aligned}
\mathsf F_X=L_X^{-1}\bigl[{}&
3\mathsf u\mathsf v\,\partial_X\mathsf u
+2\mathsf u(\partial_X\mathsf v)(\partial_X^2\mathsf u)
+2(\partial_X\mathsf u)^2\partial_X\mathsf v\\
&+\mathsf u(\partial_X^2\mathsf v)(\partial_X\mathsf u)
+\mathsf r^2\mathsf u\bigr],\\
\mathsf G_X=L_X^{-1}\bigl[{}&
3\mathsf u\mathsf v\,\partial_X\mathsf v
+2\mathsf v(\partial_X\mathsf u)(\partial_X^2\mathsf v)
+2(\partial_X\mathsf v)^2\partial_X\mathsf u\\
&+\mathsf v(\partial_X^2\mathsf u)(\partial_X\mathsf v)
-\mathsf r^2\mathsf v\bigr].
\end{aligned}
\end{equation}
Each expression in square brackets belongs to $H^{s-2}$. For instance,
$\mathsf u\partial_X\mathsf v\in H^{s-1}$ multiplies
$\partial_X^2\mathsf u\in H^{s-2}$, and the density term is treated by
expanding $\mathsf r^2\mathsf u$ around the constant background.
Thus all terms depend smoothly on
$(\xi,\mathsf u,\mathsf v,\eta_0)\in(H^s)^3\times H^{s-1}$.
The transformed system is the locally Lipschitz ODE
\begin{equation}\tag{3.39}\label{eq339}
\frac{\mathrm d}{\mathrm dt}(\xi,\mathsf u,\mathsf v)
=(\mathsf u\mathsf v,-\mathsf F_X,-\mathsf G_X),\qquad
(\xi,\mathsf u,\mathsf v)|_{t=0}=(0,u_0,v_0).
\end{equation}
If two strong Eulerian solutions have the same initial data, their lifts
solve this ODE with the same parameter $\eta_0$ and the same initial
state. Local ODE uniqueness gives equality of the lifts. Restarting the
argument at any later time proves uniqueness on their common lifespan.
The equivalence with the Eulerian equations holds first for smooth data
and then at the stated regularity by approximation and continuity of
composition by $H^s$ diffeomorphisms.

\textbf{Continuous dependence and continuation.}
The ODE theorem also gives continuous dependence on $(u_0,v_0,\eta_0)$.
Reconstruction uses
\begin{equation}\tag{3.40}\label{eq340}
u=\mathsf u\circ X^{-1},\qquad
v=\mathsf v\circ X^{-1},\qquad
\eta=(\mathsf r-\bar\rho)\circ X^{-1}.
\end{equation}
Composition and inversion are continuous at these Sobolev indices;
see \cite[Section~2 and Appendix~B]{IKTcomposition}.
It follows that the data-to-solution map is continuous in
$C(I;\mathcal H^s)$, and the equations give continuity of the time
derivatives in $C(I;\mathcal H^{s-1})$.
The same construction applies on the circle using periodic functions
and lifted periodic diffeomorphisms. Finally, restarting the local
existence argument yields the continuation alternative in $\mathcal H^s$.
This completes the proof of Theorem~\ref{3DP}.

We now prove the blow-up criterion for the three-component DP system.
\begin{proof}[Proof of Theorem~\ref{Besovcriterion}]
We first derive an a priori bound for the density. The continuity equation
and the standard $L^\infty$ transport estimate yield
\[
D^+\|\rho-1\|_\infty
\lesssim(1+\|\rho-1\|_\infty)
\|u\|_{W^{1,\infty}}\|v\|_{W^{1,\infty}}.
\]
Hence, by Gronwall's inequality,
\[
1+\|\rho(t)-1\|_\infty
\leq(1+\|\rho_0-1\|_\infty)
\exp\left(C\int_0^t\|u\|_{W^{1,\infty}}\|v\|_{W^{1,\infty}}\,\mathrm d\tau\right).
\]
Suppose, for contradiction, that $T^*<\infty$ and
\[
\int_0^{T^*}\|u\|_{W^{1,\infty}}\|v\|_{W^{1,\infty}}\,\mathrm dt<\infty.
\]
Then $\rho$ is uniformly bounded in $L^\infty$ on $[0,T^*)$.
Put $\Lambda(t)=\|u\|_{W^{1,\infty}}\|v\|_{W^{1,\infty}}$ and $w=uv$.
We first control the $H^2$ norms of the velocities through their momenta.
Multiplying the momentum equations by $m$ and $n$, respectively, and
integrating by parts gives
\[
\begin{aligned}
\frac12\frac{\mathrm d}{\mathrm dt}\|m\|_2^2
&=\frac12\int_\R w_xm^2\,\mathrm dx
 -3\int_\R vu_xm^2\,\mathrm dx-\int_\R\rho^2um\,\mathrm dx,\\
\frac12\frac{\mathrm d}{\mathrm dt}\|n\|_2^2
&=\frac12\int_\R w_xn^2\,\mathrm dx
 -3\int_\R uv_xn^2\,\mathrm dx+\int_\R\rho^2vn\,\mathrm dx.
\end{aligned}
\]
Since $u=P*m$, $v=P*n$, and $\|P\|_1=1$, the source terms satisfy
\[
\left|\int_\R\rho^2um\,\mathrm dx\right|
+\left|\int_\R\rho^2vn\,\mathrm dx\right|
\leq\|\rho\|_\infty^2(\|m\|_2^2+\|n\|_2^2).
\]
Therefore,
\begin{equation}\tag{3.41}\label{eq341}
\frac{\mathrm d}{\mathrm dt}(\|m\|_2^2+\|n\|_2^2)
\leq C(\Lambda+\|\rho\|_\infty^2)(\|m\|_2^2+\|n\|_2^2).
\end{equation}
Gronwall's inequality and the continuation assumption imply a uniform
bound for $\|u\|_{H^2}+\|v\|_{H^2}$.
For $\eta=\rho-1$, differentiation of the density equation gives
\[
(\partial_t+w\partial_x)\eta=-\rho w_x,\qquad
(\partial_t+w\partial_x)\eta_x=-2w_x\eta_x-\rho w_{xx}.
\]
Taking $L^2$ energy estimates in these two equations yields
\begin{equation}\tag{3.42}\label{eq342}
D^+\|\eta\|_{H^1}
\leq C\|w_x\|_\infty\|\eta\|_{H^1}
+C\|\rho\|_\infty\|w\|_{H^2}.
\end{equation}
It follows that
\[
M:=\sup_{t<T^*}\bigl(\|\eta(t)\|_{H^1}
+\|u(t)\|_{H^2}+\|v(t)\|_{H^2}\bigr)<\infty.
\]
We next use the dyadic transport estimates to propagate the Besov norms.
For $m_t+wm_x=-3vu_xm-\rho^2u$, the $B^0_{2,1}$ transport estimate gives
\[
D^+\|m\|_{B^0_{2,1}}
\leq C\|w\|_{H^2}\|m\|_{B^0_{2,1}}
+C\|vu_xm\|_{B^0_{2,1}}+C\|\rho^2u\|_{B^0_{2,1}}.
\]
By Bony's decomposition and the multiplier estimate in Section 2,
\[
\|vu_xm\|_{B^0_{2,1}}
\leq C\|vu_x\|_{H^1}\|m\|_{B^0_{2,1}}
\leq C_M\|m\|_{B^0_{2,1}}.
\]
Also, $\rho^2=1+2\eta+\eta^2$ and $2\eta+\eta^2$ is bounded in $H^1$.
Thus
\[
\|\rho^2u\|_{B^0_{2,1}}
\leq C_M\|u\|_{B^0_{2,1}}
\leq C_M\|m\|_{B^0_{2,1}}.
\]
The same argument applies to $n$, and consequently
\begin{equation}\tag{3.43}\label{eq343}
D^+\bigl(\|m\|_{B^0_{2,1}}+\|n\|_{B^0_{2,1}}\bigr)
\leq C_M\bigl(\|m\|_{B^0_{2,1}}+\|n\|_{B^0_{2,1}}\bigr).
\end{equation}
For the density, use the $B^1_{2,1}$ transport estimate and the tame
product bound to obtain
\[
\begin{aligned}
D^+\|\eta\|_{B^1_{2,1}}
&\leq C\|w\|_{H^2}\|\eta\|_{B^1_{2,1}}
+C\|\rho w_x\|_{B^1_{2,1}},\\
\|\rho w_x\|_{B^1_{2,1}}
&\leq C\bigl(\|\rho\|_\infty\|w_x\|_{B^1_{2,1}}
+\|w_x\|_\infty\|\eta\|_{B^1_{2,1}}\bigr).
\end{aligned}
\]
Since $w=uv$, $\|w_x\|_{B^1_{2,1}}
\leq C_M(\|u\|_{B^2_{2,1}}+\|v\|_{B^2_{2,1}})$, whence
\begin{equation}\tag{3.44}\label{eq344}
D^+\|\eta\|_{B^1_{2,1}}
\leq C_M\bigl(\|\eta\|_{B^1_{2,1}}
+\|u\|_{B^2_{2,1}}+\|v\|_{B^2_{2,1}}\bigr).
\end{equation}
The isomorphism $1-\partial_x^2$ gives
$\|u\|_{B^2_{2,1}}\simeq\|m\|_{B^0_{2,1}}$ and similarly for $v$.
Combining \eqref{eq343} and \eqref{eq344} and applying Gronwall's
inequality yields
\[
\sup_{t<T^*}\bigl(\|\eta(t)\|_{B^1_{2,1}}
+\|u(t)\|_{B^2_{2,1}}+\|v(t)\|_{B^2_{2,1}}\bigr)<\infty.
\]
The uniform $\mathcal H^2$ bound extends the solution beyond $T^*$ by
Theorem~\ref{3DP}, and the same dyadic estimates propagate its Besov
regularity on the extended interval. This contradicts maximality in the
stated Besov class and proves the criterion.
\end{proof}

For Remark 1.4, the same argument first bounds the solution in
$\mathcal H^2$. Higher Sobolev regularity then propagates on this
interval. In detail, for $\sigma\geq0$, the tame product and transport
estimates give
\[
\|af\|_{H^\sigma}\leq C_\sigma
\bigl(\|a\|_{H^1}\|f\|_{H^\sigma}
+\|a\|_{H^{\sigma+1}}\|f\|_2\bigr),
\]
\[
D^+\|f\|_{H^\sigma}\leq C_\sigma
\bigl(\|w\|_{H^2}\|f\|_{H^\sigma}
+\|w\|_{H^{\sigma+2}}\|f\|_2\bigr)+\|S_f\|_{H^\sigma}
\quad(f_t+wf_x=S_f).
\]
Apply these to $m,n$ with $\sigma=s-2$. For the density, use the
corresponding tame estimate at $r=s-1\geq1$,
\[
D^+\|\eta\|_{H^r}\leq C_r
\bigl(\|w\|_{H^2}\|\eta\|_{H^r}
+\|w\|_{H^{r+1}}\|\eta\|_{H^1}
+\|\rho w_x\|_{H^r}\bigr).
\]
The source is bounded by
$C_r(\|\rho\|_\infty\|w\|_{H^{r+1}}
+\|w_x\|_\infty\|\eta\|_{H^r})$.
Using the uniform $\mathcal H^2$ bound, we obtain
\[
D^+\|U\|_{\mathcal H^s}
\leq C_{s,M}(1+\|U\|_{\mathcal H^s}),
\]
which proves the claimed Sobolev continuation criterion.

\section{Blow-up}
We first prove the blow-up criterion in Theorem~\ref{Theorem 1.4},
then Theorem~\ref{TDPpaired}; Theorem~\ref{TGX} follows by setting
$\rho_0=\bar\rho=0$. We derive the estimates for smooth solutions and
justify them at the stated regularity at the end of the proof.

\subsection{Proof of Theorem 1.5}
\begin{proof}
\textbf{Step 1. Closed convolution estimates.}
Put
\[
g=nu_x+\rho^2-1,\quad
A=\|P*(|u|+|u_x|)\|_\infty,\quad
B=\|P*|n|\|_\infty,\quad N=\|P*|g|\|_\infty,
\]
and $S=1+A^2+B^2+N$. We first show that
\begin{equation}\tag{4.1}\label{eq41}
S(t)\leq\frac{Q_0^2}{1-15Q_0^2t}\leq2Q_0^2,
\qquad 0\leq t\leq T_1:=\frac3{184Q_0^2},\quad t<T^*.
\end{equation}
The characteristic flow is defined by
\begin{equation}\tag{4.2}\label{eq42}
X_t(t,a)=(uv)(t,X(t,a)),\qquad X(0,a)=a.
\end{equation}
Since $u,v\in H^s$, $s\geq2$, the velocity $uv$ is Lipschitz in space.
Thus $X(t,\cdot)$ is an increasing $C^1$ diffeomorphism as long as the
strong solution exists. In fact,
\[
X_a(t,a)=\exp\left(\int_0^t(uv)_x(s,X(s,a))\,\mathrm ds\right)>0.
\]
Along this flow, the momentum equation reads
\[
\frac{\mathrm d}{\mathrm dt}n(t,X(t,a))
=(-3uv_xn+\rho^2v)(t,X(t,a)).
\]
We include the density source in $g$ and therefore do not need to
estimate $\|\rho(t)-1\|_{W^{1,1}}$ for $t>0$.
Write
\[
w=uv,\quad z=u_x,\quad \alpha=vz,\quad\beta=uv_x,
\]
\[
F_*=2uvz+u(g+1),\qquad G_*=uv_xz,\qquad
J=P_x*F_*+2P*G_*,\quad H=P*F_*+2P_x*G_*.
\]
Using $P_{xx}=P-\delta_0$, the equations become
\begin{equation}\tag{4.3}\label{eq43}
\begin{cases}
u_t+(wu)_x=u\,w_x-H,\\
z_t+(wz)_x=2\beta z-J,\\
n_t+(wn)_x=-2\beta n+(g+1)v,\\
g_t+(wg)_x=-\beta g-Jn-\alpha-2\beta .
\end{cases}
\end{equation}
To verify the last cancellation, let $D=\partial_t+w\partial_x$.
The equations for $z,n$, and $\rho^2$ give
\[
Dz=-vz^2+\beta z-J,\qquad
Dn=-3\beta n+\rho^2v,\qquad
D(\rho^2)=-2(\alpha+\beta)\rho^2.
\]
Therefore,
\[
\begin{aligned}
Dg&=(Dn)z+nDz+D(\rho^2)\\
&=(-3\beta n+\rho^2v)z
+n(-vz^2+\beta z-J)-2(\alpha+\beta)\rho^2\\
&=-(\alpha+2\beta)nz-(\alpha+2\beta)\rho^2-Jn\\
&=-(\alpha+2\beta)(g+1)-Jn.
\end{aligned}
\]
Adding $w_xg=(\alpha+\beta)g$ gives the last equation in \eqref{eq43}.

For a function $h$, write $M(h)=\|P*|h|\|_\infty$.
The identities
$u=P*u-P_x*z$, $v=P*n$, and $v_x=P_x*n$ imply
$\|u\|_\infty\leq A$ and
$\|v\|_\infty,\|v_x\|_\infty\leq B$.
If $h_t+(wh)_x=R$, multiplying by a regularized sign and passing to
the limit gives
$\partial_t|h|+\partial_x(w|h|)\leq|R|$ in distributions.
Convolving with $P$ and using $|P_x|\leq P$, we obtain
\begin{equation}\tag{4.4}\label{eq44}
\partial_t(P*|h|)
\leq P*(|w||h|)+P*|R|.
\end{equation}
Moreover,
$M(F_*)\leq2BA^2+A+AN$ and $M(G_*)\leq BA^2$, so
\begin{equation}\tag{4.5}\label{eq45}
\|J\|_\infty,\ M(J),\ M(H)
\leq L:=4BA^2+A+AN.
\end{equation}
For the first equation in \eqref{eq43}, the transport term and the
source $uw_x$ contribute at most $3BA^2$, while $M(H)\leq L$.
For the second equation, the transport term and $2\beta z$ likewise
contribute at most $3BA^2$, while $M(J)\leq L$.
Adding these two pointwise convolution inequalities and then taking
spatial suprema gives $D^+A\leq6BA^2+2L$.
For the momentum, the transport term contributes $AB^2$, the stretching
term contributes $2AB^2$, and the source contributes $B(1+N)$.
For the mixed quantity, the corresponding bound is
$D^+N\leq2ABN+LB+3AB$.
Substituting the expression for $L$, we obtain
\begin{equation}\tag{4.6}\label{eq46}
\begin{aligned}
D^+A&\leq14BA^2+2A+2AN,\\
D^+B&\leq3AB^2+B+BN,\\
D^+N&\leq4A^2B^2+4AB+3ABN.
\end{aligned}
\end{equation}
The equations for $u$ and its slope can also be written in the original
nonlocal form:
\begin{equation}\tag{4.7}\label{eq47}
\begin{cases}
u_t+uvu_x+P*(3vuu_x-uu_xv_{xx}+\rho^2u)
 +2P_x*(uv_xu_x)=0,\\
u_{xt}+vu_x^2-uv_xu_x+vuu_{xx}
 +P_x*(3vuu_x-uu_xv_{xx}+\rho^2u)
 +2P*(uv_xu_x)=0.
\end{cases}
\end{equation}
Here $3vuu_x-uu_xv_{xx}+\rho^2u=2uvu_x+u(g+1)=F_*$.
In particular, along the flow,
\[
Du_x=-vu_x^2+\beta u_x-J.
\]
Set $S=1+A^2+B^2+N$. From \eqref{eq46},
\begin{equation}\tag{4.8}\label{eq48}
\begin{aligned}
D^+S\leq{}&
28BA^3+6AB^3+4A^2B^2\\
&+(4A^2+2B^2+3AB)N+(4A^2+2B^2+4AB)\\
\leq{}&15(A^2+B^2)^2+\tfrac{11}{2}(A^2+B^2)N+6(A^2+B^2)
\leq15S^2.
\end{aligned}
\end{equation}
Indeed,
$28BA^3+6AB^3=2AB(14A^2+3B^2)\leq14(A^2+B^2)^2$ and
$4A^2B^2\leq (A^2+B^2)^2$. Also,
$4A^2+2B^2+3AB\leq\frac{11}{2}(A^2+B^2)$ and
$4A^2+2B^2+4AB\leq(6A^2+B^2)$. These inequalities give the second line
of \eqref{eq48} gives the last bound.

Write
$U_0=\|u_0\|_{W^{1,1}}$, $B_0=\|n_0\|_\infty$, and
$R_0=\|\rho_0-1\|_{W^{1,1}}$.
The one-dimensional embedding
$\|f\|_\infty\leq\frac12\|f_x\|_1$ for $f\in W^{1,1}(\R)$ gives
\[
\begin{aligned}
\|\rho_0^2-1\|_1
&\leq2\|\rho_0-1\|_1
 +\|\rho_0-1\|_\infty\|\rho_0-1\|_1
\leq2R_0+\tfrac12R_0^2,\\
\|g_0\|_1&\leq B_0\|u_{0,x}\|_1+2R_0+\tfrac12R_0^2.
\end{aligned}
\]
Since $\|P\|_1=1$ and $\|P\|_\infty=\frac12$,
\begin{equation}\tag{4.9}\label{eq49}
\begin{aligned}
S(0)
&\leq1+U_0^2+B_0^2+\tfrac12B_0U_0+R_0+\tfrac14R_0^2\\
&\leq(1+U_0+B_0+\tfrac12R_0)^2=Q_0^2.
\end{aligned}
\end{equation}
Solving \eqref{eq48}, we obtain
\begin{equation}\tag{4.10}\label{eq410}
S(t)\leq\frac{S(0)}{1-15S(0)t}
\leq\frac{Q_0^2}{1-15Q_0^2t}.
\end{equation}
For $t\leq3/(184Q_0^2)$ the last expression is at most
$(184/139)Q_0^2<2Q_0^2$. This proves \eqref{eq41}.

\medskip
\noindent\textbf{Step 2. Positivity along the selected characteristic.}
Put $\nu=v_0(x_0)>0$ and $T_2=3\nu/(184Q_0^3)$.
The $v$ equation, after integration by parts, gives
\begin{equation}\tag{4.11}\label{eq411}
Dv=-P*\{3uvv_x+(v^2+v_x^2)u_x-v(g+1)\}
-P_x*(vu_xv_x).
\end{equation}
By Young's convolution inequality and the bounds from Step 1,
\[
\begin{aligned}
\|P*(3uvv_x)\|_\infty&\leq3AB^2,\\
\|P*((v^2+v_x^2)u_x)\|_\infty&\leq2B^2M(u_x)\leq2AB^2,\\
\|P*(v(g+1))\|_\infty&\leq B(N+1),\\
\|P_x*(vu_xv_x)\|_\infty&\leq B^2M(u_x)\leq AB^2.
\end{aligned}
\]
Consequently,
\[
\|Dv\|_\infty\leq6AB^2+B+BN
=B(6AB+1+N)\leq3S^{3/2}.
\]
Here $6AB\leq3(A^2+B^2)$. By \eqref{eq41},
$\|Dv\|_\infty\leq6\sqrt2\,Q_0^3$ on $[0,T_1]\cap[0,T^*)$.
Since $\nu\leq\|n_0\|_\infty\leq Q_0$, we have $T_2\leq T_1$.
Consequently,
\[
v(t,X(t,x_0))
\geq\nu-6\sqrt2\,Q_0^3t
\geq\left(1-\frac{18\sqrt2}{184}\right)\nu
>\frac{\nu}{2}
\]
for $t\leq T_2$, $t<T^*$.

\medskip
\noindent\textbf{Step 3. Riccati comparison.}
Differentiating the equation for $u$ gives the second line of
\eqref{eq47}. Evaluating it along the characteristic issued from $x_0$
removes the transport term $vuu_{xx}$. Thus, with
$f(t)=u_x(t,X(t,x_0))$ and $\nu=v_0(x_0)$,
\[
f'(t)=-v(t,X(t,x_0))f(t)^2
+\beta(t,X(t,x_0))f(t)-J(t,X(t,x_0)).
\]
For $t<\min\{T^*,T_2\}$, the estimates in Steps 1 and 2 imply
\[
v(t,X(t,x_0))\geq\nu/2,\quad
|\beta|\leq AB\leq S/2\leq Q_0^2,
\]
\[
|J|\leq A(4AB+1+N)\leq2S^{3/2}
\leq4\sqrt2\,Q_0^3<6Q_0^3.
\]
Young's inequality gives
$Q_0^2|f|\leq\frac{\nu}{4}f^2+Q_0^4/\nu$.
Combining the preceding bounds, we obtain
\[
f'\leq-\frac{\nu}{2}f^2+Q_0^2|f|+6Q_0^3
\leq-\frac{\nu}{4}f^2+\frac{Q_0^4}{\nu}+6Q_0^3
=-\frac{\nu}{4}f^2+b_1.
\]
Condition \eqref{(1.4)} is equivalent to
\[
-f(0)\geq2\sqrt{\frac{b_1}{\nu}}\,
\cothop\left(\frac{\sqrt{\nu b_1}\,T_0}{2}\right).
\]
It implies $f(0)<-2\sqrt{b_1/\nu}$.
The solution of the comparison equation
$F'=-\frac{\nu}{4}F^2+b_1$, $F(0)=f(0)$, remains negative and
decreases until its finite blow-up time. Lemma 2.10 gives that time as
\[
t_R=\frac1{\sqrt{\nu b_1}}\log\left(
\frac{f(0)-2\sqrt{b_1/\nu}}
{f(0)+2\sqrt{b_1/\nu}}\right)\leq T_0<T_2.
\]
If the strong solution persisted beyond $t_R$, the comparison inequality
would contradict the boundedness of its slope near $t_R$.
Thus $T^*\leq t_R\leq T_0$, as claimed.
\end{proof}

\setcounter{theo}{0}
\begin{rema}
For the equation for $v_x$, the corresponding mixed quantity is
$mv_x-(\rho^2-\bar\rho^2)$. The minus sign comes from the opposite
signs of the density terms in the two momentum equations.
\end{rema}

\subsection{Proof of Theorems 1.1 and 1.6}
\begin{proof}[Proof of Theorem~\ref{TDPpaired}]
\textbf{Step 1. Closed convolution estimates.}
Set $c=\bar\rho^2$ and $g=nu_x+\rho^2-c$.
Define $A,B,N,S,w,z,\alpha,\beta$ as in Step 1 of the proof of
Theorem~\ref{Theorem 1.4}, using this definition of $g$, and put
\[
F_*=2uvz+u(g+c),\qquad G_*=uv_xz,\qquad
J=P_x*F_*+2P*G_*.
\]
Following Step 1 of the proof of Theorem~\ref{Theorem 1.4}, we estimate
the convolutions of $|u|$, $|u_x|$, $|n|$, and $|g|$.
In the present setting, $M(F_*)\leq2BA^2+cA+AN$,
$M(G_*)\leq BA^2$, and consequently
$\|J\|_\infty,M(J)\leq4BA^2+cA+AN$.
Repeating the calculation leading to \eqref{eq46} gives
\begin{equation}\tag{4.12}\label{eq412}
D^+A\leq14BA^2+2\bar\rho^2A+2AN.
\end{equation}
The conservative momentum equation is
$n_t+(uvn)_x=-2uv_xn+v(g+\bar\rho^2)$, so
\begin{equation}\tag{4.13}\label{eq413}
\begin{aligned}
D^+B
&\leq\|P*(|uv||n|)\|_\infty
+\|P*(2|uv_xn|+|v||g|+\bar\rho^2|v|)\|_\infty\\
&\leq3AB^2+\bar\rho^2B+BN.
\end{aligned}
\end{equation}
Here the terms containing $nu_x$ and $\rho^2$ are combined as
$v(g+\bar\rho^2)$ before taking absolute values. The mixed quantity satisfies
\begin{equation}\tag{4.14}\label{eq414}
g_t+(uvg)_x=-uv_xg-Jn-\bar\rho^2vu_x-2\bar\rho^2uv_x,
\end{equation}
Applying \eqref{eq44} to this equation gives
\begin{equation}\tag{4.15}\label{eq415}
\begin{aligned}
D^+N
&\leq ABN+ABN+B(4BA^2+cA+AN)+3cAB\\
&=4A^2B^2+4\bar\rho^2AB+3ABN.
\end{aligned}
\end{equation}
H\"older's and Young's inequalities imply
$S(0)\leq K_\rho^2$: indeed,
$A(0)\leq\|P\|_{L^{q'}}(\|u_0\|_q+\|u_{0,x}\|_q)$,
$B(0)\leq\|P\|_{L^r}\|n_0\|_{r'}$, and
$N(0)\leq\frac12\|g_0\|_1$, with all the displayed kernel norms at most one.
Combining \eqref{eq412}, \eqref{eq413}, and \eqref{eq415}, as in
\eqref{eq48}, gives
\[
\begin{aligned}
D^+S
&\leq2A(14BA^2+2cA+2AN)
 +2B(3AB^2+cB+BN)\\
&\qquad+4A^2B^2+4cAB+3ABN\\
&\leq15(A^2+B^2)^2+\tfrac{11}{2}(A^2+B^2)N
 +6c(A^2+B^2)\leq15S^2,
\end{aligned}
\]
since $0\leq c\leq1$. Integrating this inequality gives
\begin{equation}\tag{4.16}\label{eq416}
S(t)\leq\frac{K_\rho^2}{1-15K_\rho^2t}\leq4K_\rho^2,
\qquad 0\leq t\leq\frac1{20K_\rho^2},\quad t<T^*.
\end{equation}
The definition of $\tau_\rho$ uses the smaller time
$1/(20K_\rho^3)$; this is admissible since $K_\rho\geq1$.

\textbf{Step 2. Positivity along the selected characteristic.}
Put $\nu=v_0(x_0)>0$ and $V(t)=v(t,X(t,x_0))$.
Replacing $g+1$ by $g+c$ in \eqref{eq411} and using
\eqref{eq416}, we obtain $|V'(t)|\leq3S(t)^{3/2}\leq24K_\rho^3$, so
\begin{equation}\tag{4.17}\label{eq417}
V(t)\geq\nu-24K_\rho^3t\geq\nu/2,
\qquad 0\leq t<\min\{T^*,\tau_\rho\}.
\end{equation}

\textbf{Step 3. Riccati comparison.}
On this interval, $|\beta|\leq S/2\leq2K_\rho^2$ and
$|J|\leq2S^{3/2}\leq16K_\rho^3$.
For $y(t)=-u_x(t,X(t,x_0))$, the slope equation
$Du_x=-vu_x^2+\beta u_x-J$ gives,
as long as $y\geq0$,
\[
y'\geq\frac{\nu}{2}y^2-2K_\rho^2y-16K_\rho^3
=\frac{\nu}{2}(y-\lambda_+)(y-\lambda_-),
\qquad
\lambda_\pm=\frac{2K_\rho^2\pm b_\rho}{\nu}.
\]
The initial condition implies $y(0)>\lambda_+>0$.
A first-crossing argument shows that $y$ cannot fall to $\lambda_+$.
Comparison with the scalar equality therefore gives
\[
t_R=\frac1{b_\rho}\log\left(
\frac{y(0)-\lambda_-}{y(0)-\lambda_+}\right)
=\frac1{b_\rho}\log\left(
\frac{u_{0,x}(x_0)+\lambda_-}{u_{0,x}(x_0)+\lambda_+}\right).
\]
Solving $t_R<\tau_\rho$ for $y(0)$ gives
$y(0)>(2K_\rho^2+b_\rho\cothop(b_\rho\tau_\rho/2))/\nu$, which is
the assumed slope condition.
Thus the strong solution cannot persist beyond $t_R$, which proves
the stated lifespan bound.

At $s=2$, the argument is justified without assuming pointwise derivatives
of $n$. We have $n\in L^2$, $u_x\in H^1\hookrightarrow C^0\cap L^\infty$,
and $\rho-\bar\rho\in H^1$; hence $g\in L^2$.
The velocity is Lipschitz, the conservative equations are renormalizable,
and convolution against $P$ gives the integrated form of (4.4).
On each compact time interval in the lifespan, $u_x$ and the
right-hand side of $Du_x=-vu_x^2+\beta u_x-J$ are continuous.
The distributional transport equation therefore gives
\[
u_x(t,X(t,a))=u_{0,x}(a)
+\int_0^t(-vu_x^2+\beta u_x-J)(s,X(s,a))\,\mathrm ds.
\]
This identity justifies the pointwise slope comparison at $s=2$.
\end{proof}

\begin{proof}[Proof of Theorem~\ref{TGX}]
Take $\bar\rho=0$ and $\rho_0=0$ in the preceding argument.
The density equation gives $\rho\equiv0$, so $g=nu_x$ and
$K_\rho=K$, $\tau_\rho=\tau$, $b_\rho=b$.
The equations and the comparison estimate are precisely those of
Theorem~\ref{TGX}.
\end{proof}

\section{Examples}
We give smooth initial data satisfying the blow-up conditions.
The control quantities remain bounded as the scaling parameter tends
to zero, while the negative slope at the origin tends to infinity.

\begin{example}[Initial data for Theorem 1.5]\label{exampleold}
Let
\[
\rho_0=1,\qquad v_0(x)=e^{-x^2},\qquad
u_0^\epsilon(x)=-\sqrt\epsilon\,\phi(x/\epsilon),\qquad
\phi(y)=ye^{-y^2},\quad 0<\epsilon\leq1.
\]
Then $\rho_0-1,u_0^\epsilon,v_0$ are Schwartz functions and meet all
regularity assumptions in Theorem~\ref{Theorem 1.4}.
At $x_0=0$,
$v_0(0)=1$ and $-u_{0,x}^\epsilon(0)=\epsilon^{-1/2}$.
Moreover,
\[
n_0=(3-4x^2)e^{-x^2},\qquad\|n_0\|_\infty=3,
\]
\[
\|u_0^\epsilon\|_1=\epsilon^{3/2}\|\phi\|_1=\epsilon^{3/2},
\qquad
\|u_{0,x}^\epsilon\|_1=\epsilon^{1/2}\|\phi'\|_1
=c_\phi\epsilon^{1/2},\quad
c_\phi=2\sqrt{2/e}.
\]
Consequently
$Q_0^\epsilon=4+\epsilon^{3/2}+c_\phi\epsilon^{1/2}
\leq Q_*:=5+c_\phi$.
Choose the fixed time
\[
T_0=\frac3{368Q_*^3},\qquad b_*:=Q_*^4+6Q_*^3 .
\]
Then $T_0<3/[184(Q_0^\epsilon)^3]$ and
$b_1^\epsilon\leq b_*$.
Since $\coth x\leq1+x^{-1}$ for $x>0$,
\[
2\sqrt{b_1^\epsilon}\,
\cothop\left(\frac{\sqrt{b_1^\epsilon}\,T_0}{2}\right)
\leq2\sqrt{b_*}+\frac4{T_0}.
\]
Therefore every
\[
0<\epsilon<
\min\left\{1,\left(2\sqrt{b_*}+\frac4{T_0}\right)^{-2}\right\}
\]
satisfies \eqref{(1.4)}, and its strong solution blows up by time $T_0$.
The time $T_0$ and the upper bound used for the slope threshold
are independent of $\epsilon$.
\end{example}

\begin{example}[Data for each admissible $W^{p,q}$ pairing]\label{examplepaired}
Fix any admissible $p,q,r$ in Theorem~\ref{TGX}.
Let $\phi(y)=ye^{-y^2}$ and set
\[
u_0^\epsilon(x)=-\epsilon^\alpha\phi(x/\epsilon),\qquad
v_0(x)=e^{-x^2},\qquad
\alpha=1-\frac1{2q},\quad0<\epsilon\leq1.
\]
For each fixed $\epsilon$, $u_0^\epsilon\in W^{p,q}\cap H^s$ for
every finite $s$, and $n_0=(3-4x^2)e^{-x^2}$ belongs to every
$L^{r'}$. At $x_0=0$,
\[
v_0(0)=1,\qquad-u_{0,x}^\epsilon(0)=\epsilon^{-1/(2q)}\longrightarrow\infty.
\]
The norms appearing in the criterion satisfy
\[
\|u_0^\epsilon\|_q
=\epsilon^{\,1+1/(2q)}\|\phi\|_q,\qquad
\|u_{0,x}^\epsilon\|_q
=\epsilon^{\,1/(2q)}\|\phi'\|_q,
\]
\[
\|n_0u_{0,x}^\epsilon\|_1
\leq3\epsilon^\alpha\|\phi'\|_1\longrightarrow0.
\]
Thus the corresponding $K^\epsilon$ is uniformly bounded and converges
to $(1+\|n_0\|_{r'}^2)^{1/2}$.
In particular, $\tau^\epsilon$ is bounded away from zero, while
$b^\epsilon$ and the entire right-hand side of the slope condition
remain bounded. For all sufficiently small $\epsilon$,
Theorem~\ref{TGX} applies. Writing $y_\epsilon=-u_{0,x}^\epsilon(0)=\epsilon^{-1/(2q)}$ and
using $\lambda_+^\epsilon-\lambda_-^\epsilon=2b^\epsilon$, we obtain
\[
T_\epsilon^*\leq\frac1{b^\epsilon}
\log\left(1+\frac{2b^\epsilon}{y_\epsilon-\lambda_+^\epsilon}\right)
\leq\frac2{y_\epsilon-\lambda_+^\epsilon}
\leq4\epsilon^{1/(2q)}
\]
for all sufficiently small $\epsilon$, since $\lambda_+^\epsilon$
is bounded. Here we used $\log(1+z)\leq z$ for $z\geq0$.

For the three-component system, take $\rho_0=\bar\rho$ with any
$\bar\rho\in[0,1]$.
Then $\rho_0-\bar\rho=0$, $\rho_0^2-\bar\rho^2=0$, and
$g_0=n_0u_{0,x}^\epsilon$, so the same argument applies to
Theorem~\ref{TDPpaired}.
The choice $\bar\rho=0$ gives exactly the GX reduction.
Each $u_0^\epsilon$ belongs to $W^{p,q}$; its full $W^{p,q}$ norm
need not be bounded uniformly in $\epsilon$.
\end{example}

\begin{example}[A nonconstant perturbation of the density]
Keep $u_0^\epsilon,v_0$ from Example~\ref{examplepaired}, fix $\delta>0$,
and let
\[
\bar\rho=1,\qquad \rho_0(x)=1+\delta e^{-x^2}.
\]
Then $\rho_0-1$ is Schwartz, and
\[
\rho_0^2-1=2\delta e^{-x^2}+\delta^2e^{-2x^2},\qquad
\|\rho_0^2-1\|_1=2\delta\sqrt\pi+\delta^2\sqrt{\pi/2}.
\]
Consequently,
\[
\|g_0^\epsilon\|_1
\leq3\epsilon^\alpha\|\phi'\|_1
+2\delta\sqrt\pi+\delta^2\sqrt{\pi/2}.
\]
For fixed $\delta$, $K_\rho^\epsilon$ remains bounded, whereas
$-u_{0,x}^\epsilon(0)\to\infty$.
Theorem~\ref{TDPpaired} therefore applies for all sufficiently small
$\epsilon$.
\end{example}

The three examples verify the hypotheses of Theorems~\ref{Theorem 1.4},
\ref{TGX}, and \ref{TDPpaired}, including the case of a nonconstant
density perturbation.

\section{Declarations}
\noindent\textbf{Funding.} Z. Yin was supported by the National Natural
Science Foundation of China (No.~12571261).

\noindent\textbf{Author contributions.} S. Liu and Z. Yin contributed
equally to this work.

\noindent\textbf{Data availability statement.}
Data sharing is not applicable to this article, as no datasets were
generated or analysed during the study.

\noindent\textbf{Conflict of interest.} The authors declare that they have
no conflict of interest.

\end{document}